\documentclass[11pt]{article}
\usepackage[utf8]{inputenc}

\usepackage{amsfonts,amsmath,amssymb,amsthm}
\usepackage{graphicx}
\usepackage{indentfirst}
\usepackage{mathtools}

\usepackage{tocloft}
\usepackage{bm}
\usepackage{xcolor}

\usepackage{tikz}
\usetikzlibrary{shapes,chains,scopes,positioning,backgrounds,fit,shadows,calc,arrows.meta}
\usepackage{caption}
\usepackage[urlcolor=blue]{hyperref}

\newtheorem{theorem}{Theorem}

\newtheorem{corollary}[theorem]{Corollary}

\newtheorem{Fact}{Fact}

\theoremstyle{definition}
\theoremstyle{definition}\newtheorem*{case 1}{Case 1}
\theoremstyle{definition}\newtheorem*{case 2}{Case 2}
\theoremstyle{definition}\newtheorem*{case 3}{Case 3}
\theoremstyle{definition}\newtheorem*{case 4}{Case 4}
\theoremstyle{definition}\newtheorem*{case 5}{Case 5}
\theoremstyle{definition}\newtheorem*{subcase 1.1}{Subcase 1.1}
\theoremstyle{definition}\newtheorem*{subcase 1.2}{Subcase 1.2}
\theoremstyle{definition}\newtheorem*{subcase 1.3}{Subcase 1.3}
\theoremstyle{definition}\newtheorem*{subcase 3.1}{Subcase 3.1}
\theoremstyle{definition}\newtheorem*{subcase 3.2}{Subcase 3.2}
\theoremstyle{definition}\newtheorem*{subcase 3.3}{Subcase 3.3}
\theoremstyle{definition}\newtheorem*{subcase 1.2.1}{Subcase 1.2.1}
\theoremstyle{definition}\newtheorem*{subcase 1.2.2}{Subcase 1.2.2}
\theoremstyle{definition}\newtheorem*{subcase 1.2.3}{Subcase 1.2.3}
\theoremstyle{definition}\newtheorem*{subcase 2.2.1}{Subcase 2.2.1}
\theoremstyle{definition}\newtheorem*{subcase 2.2.2}{Subcase 2.2.2}
\theoremstyle{definition}\newtheorem*{subcase 2.2.3}{Subcase 2.2.3}
\theoremstyle{definition}\newtheorem*{subcase 2.1}{Subcase 2.1}
\theoremstyle{definition}\newtheorem*{subcase 2.1.1}{Subcase 2.1.1}
\theoremstyle{definition}\newtheorem*{subcase 2.1.2}{Subcase 2.1.2}
\theoremstyle{definition}\newtheorem*{subcase 2.1.3}{Subcase 2.1.3}
\theoremstyle{definition}\newtheorem*{subcase 2.2}{Subcase 2.2}
\theoremstyle{definition}\newtheorem*{subcase 2.3}{Subcase 2.3}
\theoremstyle{definition}\newtheorem*{subcase 2.4}{Subcase 2.4}
\theoremstyle{definition}\newtheorem*{subcase 2.4.1}{Subcase 2.4.1}
\theoremstyle{definition}\newtheorem*{subcase 2.4.2}{Subcase 2.4.2}
\theoremstyle{definition}\newtheorem*{subcase 2.4.3}{Subcase 2.4.3}
\theoremstyle{definition}\newtheorem*{subcase 2.4.3.1}{Subcase 2.4.3.1}
\theoremstyle{definition}\newtheorem*{subcase 2.4.3.2}{Subcase 2.4.3.2}
\theoremstyle{definition}\newtheorem*{subcase 2.4.3.3}{Subcase 2.4.3.3}
\theoremstyle{definition}\newtheorem*{subcase 2.4.3.4}{Subcase 2.4.3.4}
\theoremstyle{definition}\newtheorem*{subcase 2.4.3.5}{Subcase 2.4.3.5}
\theoremstyle{definition}\newtheorem*{subcase 2.5}{Subcase 2.5}
\theoremstyle{definition}\newtheorem*{subcase 2.5.1}{Subcase 2.5.1}
\theoremstyle{definition}\newtheorem*{subcase 2.5.2}{Subcase 2.5.2}
\theoremstyle{definition}\newtheorem*{subcase 2.5.3}{Subcase 2.5.3}
\theoremstyle{definition}\newtheorem*{subcase 2.5.4}{Subcase 2.5.4}
\theoremstyle{definition}\newtheorem*{subcase 2.5.5}{Subcase 2.5.5}
\theoremstyle{definition}\newtheorem*{subcase 2.6}{Subcase 2.6}
\theoremstyle{definition}\newtheorem*{subcase 3.4}{Subcase 3.4}
\theoremstyle{definition}\newtheorem*{subcase 4.1}{Subcase 4.1}
\theoremstyle{definition}\newtheorem*{subcase 4.2}{Subcase 4.2}
\theoremstyle{definition}\newtheorem*{subcase 4.3}{Subcase 4.3}
\theoremstyle{definition}\newtheorem*{subcase 4.4}{Subcase 4.4}
\theoremstyle{definition}\newtheorem*{subcase 4.5}{Subcase 4.5}
\theoremstyle{definition}\newtheorem*{subcase 5.1}{Subcase 5.1}
\theoremstyle{definition}\newtheorem*{subcase 5.2}{Subcase 5.2}
\theoremstyle{definition}\newtheorem*{subcase 5.3}{Subcase 5.3}
\theoremstyle{definition}\newtheorem*{subcase 5.4}{Subcase 5.4}
\theoremstyle{definition}\newtheorem*{subcase 5.5}{Subcase 5.5}
\theoremstyle{definition}\newtheorem*{subcase 5.6}{Subcase 5.6}
\theoremstyle{definition}
\theoremstyle{definition}\newtheorem*{claim 1}{Claim 1}
\theoremstyle{definition}\newtheorem*{claim 2}{Claim 2}
\theoremstyle{definition}\newtheorem*{claim 3}{Claim 3}

\title{Arc-disjoint Steiner Cycles in Digraphs}
\author{Jie Bai$^{1}$,  Yuefang Sun$^{2,}$\footnote{Corresponding author.},  Chuchu Wang$^{3}$, 
 Shanshan Yu$^{4}$
\\
$^{1}$ School of Mathematics and Statistics,
Ningbo University\\
Zhejiang 315211,  China, jiebai1211@163.com\\
$^{2}$ School of Mathematics and Statistics,
Ningbo University\\
Zhejiang 315211, China, sunyuefang@nbu.edu.cn\\
$^{3}$ School of Mathematics and Statistics,
Ningbo University\\
Zhejiang 315211,  China, wcc7795@163.com\\
$^{4}$ School of Mathematics and Statistics,
Ningbo University\\
Zhejiang 315211,  China, yushanshan33hh@163.com\\
}

\date{}
\begin{document}
		\maketitle
		\begin{abstract}
Let $D=(V(D), A(D))$ be a digraph of order $n$ and let $S\subseteq V(D)$ with $2\leq |S|\leq n$.  A directed cycle $C$ of $D$ is called a directed $S$-Steiner cycle  (or, an $S$-cycle for short) if $S\subseteq V(C)$. Steiner cycles have applications in reliable designs for telecommunication and transportation networks. Two $S$-cycles are called arc-disjoint if they have no common arcs. We use $\lambda_{S}^{c}(D)$ to denote the maximum number of pairwise arc-disjoint $S$-cycles in $D$. The directed cycle $k$-arc-connectivity of $D$ is defined as $$\lambda_{k}^{c} (D)=\min\left \{ \lambda _{S}^{c}(D)\mid S\subseteq V(D),\left | S \right | =k,2\le k\le n  \right \}.$$
			
In this paper, we determine the complexity for $\lambda_{S}^{c} (D)$ on Eulerian digraphs, planar digraphs and symmetric digraphs. We also obtain exact values of $\lambda_{k}^{c} (D)$ on complete digraphs, complete bipartite digraphs and complete regular multipartite digraphs.

\vspace{0.2cm}		 
\end{abstract}
		
\section{Introduction}

\subsection{Motivation, terminology and notation}
All terminology and notation concerning graph theory and algorithmic complexity that are not defined here can be found in \cite{Bang-Jensen09, J.A. Bondy}. Unless stated otherwise, all digraphs considered in this paper are simple, that is, have no loops or parallel arcs. Let $[n]=\{1,2,\cdots,n\}$ and let $deg^+_{D}(v)$, $deg^-_{D}(v)$ be the out-degree and in-degree of a vertex $v$ in the digraph $D$, respectively. For a directed path $P = x_0x_1\cdots x_k$ from $x_0$ to $x_k$, simply denoted by $x_0Px_k$, we call it an {\em $x_0-x_k$ path}. We use \emph{$(v_0,v_k)$-path} to denote an undirected path connecting $v_0$ and $v_k$. A \emph{(directed) Hamiltonian cycle} in a (directed) graph $G$ is a (directed) cycle that visits each vertex of $G$ exactly once. A digraph $D$ is \emph{Eulerian} if its underlying undirected graph is connected and each vertex $v \in V(D)$ satisfies $deg^+_D(v)=deg^-_D(v)$. A digraph $D$ is \emph{planar} if its underlying undirected graph is planar, and \emph{symmetric} if for every arc $xy\in A(D)$, the reverse arc $yx$ also belongs to $A(D)$.

Packing problems are central in graph theory and combinatorial optimization. A classical example is the {\sc Steiner Tree Packing} problem. 
Given an undirected graph $G$ and a terminal set $S\subseteq V(G)$ with $|S|\geq 2$, an \emph{$S$-Steiner tree} (or simply an \emph{$S$-tree}) is a tree $T$ of $G$ with $S\subseteq V(T)$. Two $S$-trees are called \emph{edge-disjoint} if they share no common edge. The basic Steiner tree packing problem asks for the largest collection of pairwise edge-disjoint $S$-trees in $G$. This problem is motivated by both structural graph theory and applications in network design, including VLSI circuit design~\cite{Grotschel-Martin-Weismantel, Sherwani} and Internet Domain~\cite{Li-Mao5}. It is also closely connected with Kriesell's conjecture, which predicts the existence of many edge-disjoint Steiner trees under suitable edge-connectivity assumptions on the terminal set.

Several variants of Steiner tree packing have been studied. In digraphs, one may require the $S$-trees to be edge-disjoint or internally disjoint. In directed graphs, there are several possible analogues, depending on which directed structure to be packed. These include the directed Steiner tree packing problem~\cite{Cheriyan-Salavatipour, Sun-Yeo}, the directed Steiner path packing problem~\cite{Sun-Zhang}, the strong subgraph packing problem~\cite{Sun-Gutin2, Sun-Gutin-Yeo-Zhang, Sun-Gutin-Zhang}, and the directed Steiner cycle packing problem~\cite{Wang-Sun}. Such problems are often called directed Steiner type packing problems, and we refer the reader to \cite{Sun-book} for a recent monograph of this area. 

This paper focuses on the cyclic version. Let $D=(V,A)$ be a digraph and $S\subseteq V$ with $|S|\ge 2$. A directed cycle $C$ of $D$ is called an \emph{$S$-Steiner cycle} (or simply an \emph{$S$-cycle}) if $S\subseteq V(C)$. Two $S$-cycles $C_1, C_2$ are \emph{arc-disjoint} if they have no common arc and they are \emph{internally disjoint} if they are arc-disjoint and their common vertex set is exactly $S$. 
Directed Steiner cycle packing has two natural versions.

{\sc Arc-disjoint directed Steiner cycle packing~(ADSCP):}
given a digraph $D$ and a terminal set $S\subseteq V(D)$, 
find a largest collection of pairwise arc-disjoint $S$-cycles.

{\sc Internally disjoint directed Steiner cycle packing~(IDSCP):} 
given a digraph $D$ and a terminal set $S\subseteq V(D)$, 
find a largest collection of pairwise internally disjoint $S$-cycles.

Directed Steiner cycle packing is related to several classical and recent topics, but it has its own distinct features. 
When $S=V(D)$, an $S$-cycle is a Hamiltonian cycle, so the problem contains Hamiltonian cycle packing as a special case. Since every directed cycle is strong, an $S$-cycle is also an $S$-strong subgraph. Hence directed Steiner cycle packing may be viewed as a cyclic restriction of strong subgraph packing. 

Recently, Sun and Jin~\cite{Sun-Jin1} studied the complexity of directed Steiner cycle packing. They proved that, 
for any fixed integers $k \geq 2$ and $\ell \geq 1$, both ADSCP and IDSCP are NP-complete on general digraphs. 
They also proved that IDSCP remains NP-complete on Eulerian digraphs, while it becomes polynomial-time solvable on symmetric digraphs when $k$ and $\ell$ are fixed. 

To measure the maximum possible size of such packings, we denote by $\kappa_S^c(D)$ and $\lambda_S^c(D)$ the maximum numbers of pairwise internally disjoint and arc‑disjoint $S$-cycles in $D$, respectively. 
Closely related to these local parameters, Wang and Sun~\cite{Wang-Sun} introduced the \emph{directed cycle $k$-connectivity} of a digraph $D$, defined as
$$\kappa _{k}^{c} (D)=\min\left \{ \kappa _{S}^{c}(D)\mid S\subseteq V(D),\left | S \right | =k,2\le k\le n  \right \}.$$
This parameter characterizes the minimum number of internally disjoint directed cycles that can be forced through any prescribed $k$ vertices of $D$.
Furthermore, they fully characterized the directed cycle $k$-connectivity for complete digraphs $\overleftrightarrow{K}_{n}$ and complete regular bipartite digraphs $\overleftrightarrow{K}_{a,a}$. For $\overleftrightarrow{K}_{n}$, they provided a sharp lower bound and exact values when $k\in\{2,3,4,6\}$. For $\overleftrightarrow{K}_{a,a}$, they determined the exact value for every admissible $k$.
Similarly, the \emph{directed cycle $k$-arc-connectivity} can be defined as
$$\lambda _{k}^{c} (D)=\min\left \{ \lambda _{S}^{c}(D)\mid S\subseteq V(D),\left | S \right | =k,2\le k\le n  \right \}.$$
We call $\kappa_k^c(D), \lambda_k^c(D)$ directed cycle
connectivity. The Steiner cycle packing problem is closely related to the directed cycle connectivity, but they differ in focus. The Steiner cycle packing problem studies a local property of a digraph, since the vertex subset $S$ is given in advance. In contrast, directed cycle connectivity characterizes a global property of a digraph. For a fixed positive integer $k\geq 2$, one first computes $\kappa_S^c(D)$ and $\lambda_S^c(D)$ for every subset $S\subseteq V(D)$ with $|S|=k$, and then takes the minimum over all such subsets. The resulting values are the directed cycle connectivity $\kappa_k^c(D)$ and $\lambda_k^c(D)$, respectively.

\subsection{Our results}

This paper studies the arc-disjoint directed Steiner cycle packing problem (ADSCP). We investigate both the computational complexity of deciding whether $\lambda_{S}^{c}(D)\geq \ell$ and the exact value of the global parameter $\lambda_k^c(D)$ on several classes of digraphs.


In Section 2, we investigate the complexity of $\lambda_{S}^{c}(D)$ on Eulerian digraphs, planar digraphs and symmetric digraphs. More precisely, we show that the problem remains NP-complete for Eulerian digraphs when $k\geq 3$ is fixed and $\ell$ is part of the input (Theorem~\ref{Eulerian digraph}), and for planar digraphs when $k\geq 2$ is fixed and $\ell$ is part of the input (Theorem~\ref{planar digraph}). We also prove that, when 
$\ell\geq 1$ is fixed and $k$ is part of the input, the problem is NP‑complete for all three classes (see Theorems~\ref{symmetric digraph}, \ref{planar digraph2} and~\ref{Eulerian digraph2}). In contrast, when $k \geq 2$ is fixed and $\ell\in\{1, 2\}$, the problem can be solved in polynomial time for symmetric digraphs (Theorem~\ref{symmetric digraph2}). 

It is worth noting that, when $\ell \ge 1$ is fixed and $k$ is part of the input, our reductions for Eulerian digraphs, planar digraphs and symmetric digraphs yield more than arc-disjoint $S$-cycles. The constructed $S$-cycles can also be chosen to be internally disjoint. Hence the same reductions also give the corresponding NP-completeness results for IDSCP on these three classes of digraphs. 


In Section 3, we determine exact values of $\lambda^c_k(D)$ on complete digraph $\overleftrightarrow{K}_{n}$,
complete regular multipartite digraph $\overleftrightarrow{K}_{w,w,\dots,w}$ ($l$ times, denoted simply by $\overleftrightarrow{K}_{[w]^l}$)
and complete bipartite digraph $\overleftrightarrow{K}_{t,z}$. These results are summarized in Theorems~\ref{1} and~\ref{2}.

\section{Computational complexity for $\lambda_{S}^{c}(D)$ on Eulerian digraphs, planar digraphs and symmetric digraphs}

To prove that deciding whether $\lambda_S^c(D) \geq \ell$ is NP-complete on Eulerian digraphs for fixed $k$ and with $\ell$ part of the input, we reduce from the classic NP-complete problem: the general {\sc weak $2$-linkage} problem. It is defined as follows: given a digraph $D$ and two pairs $(s_1,t_1)$, $(s_2,t_2)$ of distinct vertices (all four vertices are distinct), deciding whether $D$ contains a pair of arc-disjoint paths $P_1,P_2$ with $P_i$ an $s_i–t_i$ path for $i\in [2]$.



\begin{theorem}\label{S. Fortune}\cite{{S. Fortune}}\label{general weak 2-linkage}
    The general {\sc weak 2-linkage} problem is NP-complete. 
\end{theorem}

In the reduction from the weak $2$-linkage problem to our problem, we need to analyze the flow in the constructed digraph in order to extract the required arc-disjoint paths. To this end, we need the following \emph{Flow Decomposition Theorem}, and then we establish the NP-completeness of deciding $\lambda_S^c(D) \ge \ell$ on Eulerian digraphs.

\begin{theorem}[Flow Decomposition Theorem]\label{Flow Decomposition Theorem}\cite{Bang-Jensen09}
    Every flow $x$ in a network $\mathcal{N}$ can be represented as the arc sum of some path and cycle flows
    $$f(P_1), f(P_2), \dots, f(P_\alpha), f(C_1), \dots, f(C_\beta)$$
    satisfying the following two properties:
    \begin{description}
        \item[(i)] Every directed path $P_i\ (i\in[\alpha])$ with positive flow connects a source vertex to a sink vertex.
        \item[(ii)] $\alpha + \beta \le |V(D)| + |A(D)|$ and $\beta \le |A(D)|$.
    \end{description}
\end{theorem}

\begin{theorem}\label{Eulerian digraph}
  Fix an integer $k \geq 3$ and let $\ell$ be part of the input. Then, given an Eulerian digraph $D$ and a subset $S \subseteq V(D)$ with $|S| = k$, deciding whether $\lambda_S^c(D) \ge \ell$ is NP-complete.
\end{theorem}

\begin{proof}
    Clearly, the problem is in NP. To show the NP-hardness of it, we reduce from the general {\sc weak 2-linkage} problem, which is known to be NP-complete by Theorem~\ref{S. Fortune}.  
    
    Let $[G;s_1, t_1, s_2, t_2]$ be an instance of the general {\sc weak 2-linkage} problem, where $G$ is a digraph and $(s_1,t_1),(s_2,t_2)$ are two terminal pairs of distinct vertices. By adding the two arcs $\overrightarrow{t_1s_1}$ and $\overrightarrow{t_2s_2}$ to $G$, we get a directed multigraph, denoted by $G+H$.

    First, we construct an Eulerian directed multigraph $G'$ from $G+H$. Let $p$ denote the sum of $\deg^+_{G+H}(v)-\deg^-_{G+H}(v)$ over all vertices $v$ for which this difference is positive.

If $p \geq 1$, we add two new vertices $s$ and $t$ to $G+H$. For every $v \in V(G+H)$ with $\deg^+_{G+H}(v) > \deg^-_{G+H}(v)$, we add $\deg^+_{G+H}(v)-\deg^-_{G+H}(v)$ parallel arcs $\overrightarrow{sv}$. For every $v\in V(G+H)$ with $\deg^-_{G+H}(v)>\deg^+_{G+H}(v)$, we add $\deg^-_{G+H}(v)-\deg^+_{G+H}(v)$ parallel arcs $\overrightarrow{vt}$. Finally, add $p$ parallel arcs $\overrightarrow{ts}$. Let $G'$ be the resulting digraph, which is Eulerian (see Figure~\ref{figure1}). If $p=0$, then $G+H$ is already Eulerian, and we simply set $G'=G+H$ (without vertices $s,t$).

In both cases, by the Eulerian property and Theorem~\ref{Flow Decomposition Theorem}, the existence of a pair of arc-disjoint $s_1-t_1$ path and $s_2-t_2$ path in $G$ is equivalent to the existence of $p+2$ arc-disjoint paths in $G'$: one $s_1-t_1$ path, one $s_2-t_2$ path, and $p$ $s-t$ paths.

\begin{figure}[htb]
    \centering 
    \makebox[0pt]{
        \begin{tikzpicture}[scale=1.0, >=latex, every node/.style={font=\normalsize}] 

            \filldraw[black] (1.5, -2.5) circle (1.8pt) node [below=2pt] {$s_{1}$};
            \filldraw[black] (3.0, -2.5) circle (1.8pt) node [below=2pt] {$t_{1}$};
            \filldraw[black] (4.5, -2.5) circle (1.8pt) node [below=2pt] {$s_{2}$};
            \filldraw[black] (6.0, -2.5) circle (1.8pt) node [below=2pt] {$t_{2}$};

            \filldraw[black] (-1.2, -2.5) circle (2.2pt) node [left=2pt, font=\large] {$s$};
            \filldraw[black] (8.2, -2.5)  circle (2.2pt) node [right=2pt, font=\large] {$t$};

            \def\boxleft{0.0}
            \def\boxright{7}
            \def\boxtop{-1.6}
            \def\boxbottom{-3.4}
            \draw[line width=1pt] (\boxleft, \boxtop) rectangle (\boxright, \boxbottom);

            \coordinate (box_left_mid) at (\boxleft, -2.5);
            \coordinate (box_right_mid) at (\boxright, -2.5);

            \draw[thick, -latex, line width=2pt] (-1.2, -2.5) -- (box_left_mid);
            \draw[thick, -latex, line width=2pt] (box_right_mid) -- (8.2, -2.5);
            \draw[thick, black, ->, line width=2pt] (8.2, -2.5) to[out=120, in=60, looseness=0.9] (-1.2, -2.5);

            \draw[black, thick, ->] (3.0, -2.5) to[out=150, in=30, looseness=0.8] (1.5, -2.5);
            \draw[black, thick, ->] (6.0, -2.5) to[out=150, in=30, looseness=0.8] (4.5, -2.5);

            \node at(3.6, -4) {$G'$};
        \end{tikzpicture}
    }
    \caption{The Eulerian digraph $G'$.}
    \label{figure1}
\end{figure}
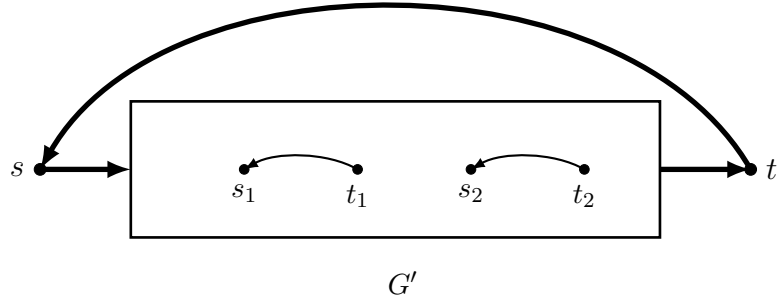


Secondly, we construct the digraph $D$ from $G'$ and $S$, where $S=\{x_1,\dots,x_k\}$.

\medskip

\noindent\textbf{Case 1: $p \geq 1$.} 
Let
\[
V(D')=V(G')\cup S,
\]
and
\[
A(D')=A(G)\cup A'\cup
\{\overrightarrow{x_{k-1}s_1},\,
\overrightarrow{t_1x_k},\,
\overrightarrow{x_ks_2},\,
\overrightarrow{t_2x_1},\,
\overrightarrow{x_{k-2}s},\,
\overrightarrow{tx_{k-1}}\}
\cup
\{\overrightarrow{x_ix_{i+1}}\mid i\in[k]\},
\]
where $A'=\{\overrightarrow{sv},\overrightarrow{vt}\mid v\in V(G+H)\}$ and $x_{k+1}=x_1$. 
Furthermore, we add $p+1$ additional copies of $\overrightarrow{x_ix_{i+1}}$ for each $i\in[k-3]$, one copy of $\overrightarrow{x_{k-2}x_{k-1}}$, $p$ copies of $\overrightarrow{x_{k-1}x_k}$ and $\overrightarrow{x_kx_1}$, and $p-1$ additional copies of $\overrightarrow{x_{k-2}s}$ and $\overrightarrow{tx_{k-1}}$. 
To avoid parallel arcs, we insert a new vertex $z_{i,i+1}^j$ into each $\overrightarrow{x_ix_{i+1}}$, where $j\in[p+2]$ if $i\in[k-3]$, $j\in[2]$ if $i=k-2$, and $j\in[p+1]$ if $i\in\{k-1,k\}$. 
We also insert $u_i$ ($i\in[p]$) into each $\overrightarrow{x_{k-2}s}$, $w_i$ ($i\in[p]$) into each $\overrightarrow{tx_{k-1}}$, and a distinct new vertex into each arc of $A'$, where the set of these new vertices is denoted by $V_{A'}$.
The resulting digraph, denoted by $D$, is Eulerian (see Figure~\ref{figure2}).

\begin{figure}[htb]
    \centering 
    \makebox[0pt]{
        \begin{tikzpicture}[scale=1.0, >=latex, every node/.style={font=\small}]
    
            \filldraw[black]    (0, 0)  circle (2pt)  node [anchor=north east] {$x_{1}$};
            \filldraw[black]    (0.9, 0)  circle (2pt)  node [anchor=south] {$z^{j}_{1,2}$};
            \filldraw[black]    (1.8, 0)  circle (2pt)  node [anchor=north] {$x_{2}$};
            \filldraw[black]    (3.6, 0)  circle (2pt)  node [anchor=north, xshift=6pt] {$x_{k-2}$};
            \filldraw[black]    (4.5, 0)  circle (2pt)  node [anchor=south] {$z^{j}_{k-2,k-1}$};

            \filldraw[black]    (5.4, 0)  circle (2pt)  node [anchor=north, xshift=3pt, yshift=-5pt] {$x_{k-1}$};
            \filldraw[black]    (6.3, 0)  circle (2pt)  node [anchor=south] {$z^{j}_{k-1,k}$};
            \filldraw[black]    (7.2, 0)  circle (2pt)  node [anchor=north west] {$x_{k}$};

            \coordinate (z_j_kk1) at (3.6, {-9/80*3.6*(3.6-7.2)});
            \filldraw[black]    (z_j_kk1)  circle (2pt)  node [anchor=south] {$z^{j}_{k,k+1}$};

            \filldraw[black]    (1.8, -4.25)  circle (2pt)  node [anchor=north east, yshift=1pt, xshift=2pt] {$s_{1}$};
            \filldraw[black]    (3, -4.25)  circle (2pt)  node [anchor=north east, yshift=1pt, xshift=2pt] {$t_{1}$};
            \filldraw[black]    (4, -4.25)  circle (2pt)  node [anchor=north east, yshift=1pt, xshift=2pt] {$s_{2}$};
            \filldraw[black]    (5.2, -4.25)  circle (2pt)  node [anchor=north west, yshift=1pt, xshift=-1pt] {$t_{2}$};

            \filldraw[black]    (-1.8, -2.4)  circle (2pt)  node [anchor=north east] {$s$};
            \filldraw[black]    (9.0, -2.4)  circle (2pt)  node [anchor=north west] {$t$};

            \filldraw[black]    (0.9, -1.25)  circle (2pt)  node [anchor=north west] {$u_{i}$};
            \filldraw[black]    (7.2, -1.25)  circle (2pt)  node [anchor=north east] {$w_{i}$};

            \def\boxleft{1}    
            \def\boxright{6.2}   
            \def\boxtop{-3}     
            \def\boxbottom{-5.5}  
            \draw[line width=1pt] (\boxleft, \boxtop) rectangle (\boxright, \boxbottom) node [anchor=west] {$G$};

            \coordinate (box_left_mid) at (\boxleft, {(\boxtop+\boxbottom)/2});
            \coordinate (box_right_mid) at (\boxright, {(\boxtop+\boxbottom)/2});

            \draw[densely dotted, line width=1.2pt]      (2.3, 0) -- (3.1, 0);
            \draw[thick, -latex]      (0, 0) -- (0.9, 0);
            \draw[thick, -latex]      (0.9, 0) -- (1.8, 0);
            \draw[thick, -latex]      (3.6, 0) -- (4.5, 0);
            \draw[thick, -latex]      (4.5, 0) -- (5.4, 0);
            \draw[thick, -latex]      (5.4, 0) -- (6.3, 0);
            \draw[thick, -latex]      (6.3, 0) -- (7.2, 0);
            \draw[thick, -latex]      (5.2, -4.25) -- (0, 0);
            \draw[thick, -latex]      (5.4, 0) --  (1.8, -4.25);
            \draw[thick, -latex]      (3, -4.25) -- (7.2, 0);
            \draw[thick, -latex]      (7.2, 0) -- (4, -4.25);
            \draw[thick, -latex]      (3.6, 0)  -- (0.9, -1.25);
            \draw[thick, -latex]      (0.9, -1.25) -- (-1.8, -2.4);
  
            \draw[thick, -latex]      (9.0, -2.4) -- (7.2, -1.25);
            
            \draw[thick, -latex]      (7.2, -1.25) -- (5.4, 0);

            \draw[thick, -latex, line width=2pt] (-1.8, -2.4) -- (box_left_mid);
            \draw[thick, -latex, line width=2pt] (box_right_mid) -- (9.0, -2.4);

            \draw[thick, -latex] plot[smooth, domain=7.2:3.6] (\x, {-9/80*\x*(\x-7.2)});
            \draw[thick, -latex] plot[smooth, domain=3.6:0] (\x, {-9/80*\x*(\x-7.2)});

            \node at(3.6, -6) {$D$};
        \end{tikzpicture}
    }
    \caption{The Eulerian digraph $D$.}
    \label{figure2}
\end{figure}
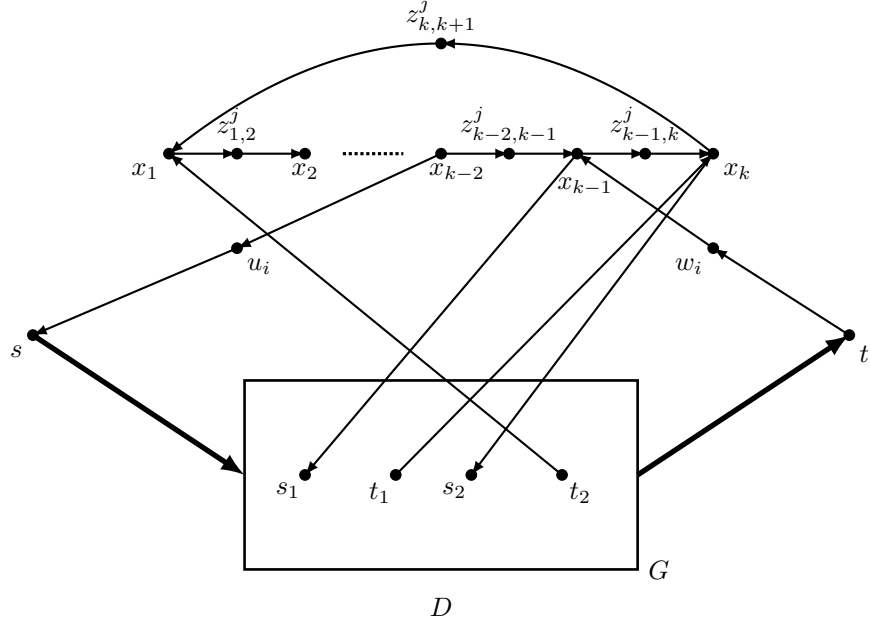

\medskip

\noindent\textbf{Case 2: $p=0$.} 
Since $G'=G+H$ contains no vertices $s,t$, we set
\[
V(D')=V(G)\cup S,
\]
and
\[
A(D')=A(G)\cup
\{\overrightarrow{x_{k-1}s_1},\,
\overrightarrow{t_1x_k},\,
\overrightarrow{x_ks_2},\,
\overrightarrow{t_2x_1}\}
\cup
\{\overrightarrow{x_ix_{i+1}}\mid i\in[k]\}.
\]
For each $\overrightarrow{x_ix_{i+1}}$ ($i\in[k-2]$), we add one additional copy, and insert different new vertices into each $\overrightarrow{x_ix_{i+1}}$. This yields an Eulerian digraph $D$.

We will show that $\lambda_S^c(D)\geq \ell= p+2$ if and only if there exist a pair of arc-disjoint $s_1-t_1$ path and $s_2-t_2$ path in $G$. Then, together with Theorem~\ref{general weak 2-linkage}, this completes the proof.

Suppose $G$ contains an $s_1-t_1$ path (denoted by $P_1$) and an $s_2-t_2$ path (denoted by $P_2$), which are arc-disjoint.
If $p\geq1$, we also have $p$ arc-disjoint $s-t$ paths (denoted by $P_j$ for $3\leq j\leq p+2$) in $D[V(G)\cup\{s,t\}\cup V_{A'}]$. 
Thus,
\[
\begin{aligned}
C_1 &= x_1z^{1}_{1,2}x_2\dots x_{k-1}s_1P_1t_1x_{k}z^{1}_{k,k+1}x_1, \\
C_2 &= x_1z^{2}_{1,2}x_2\dots x_{k-1}z^{1}_{k-1,k}x_{k}s_2P_2t_2x_1, \\
C_j &= x_1z^{j}_{1,2}x_2\dots x_{k-2}u_{j-2}sP_jtw_{j-2}x_{k-1}z^{j-1}_{k-1,k}x_kz^{j-1}_{k,k+1}x_1,
\end{aligned}
\]
are $p+2$ pairwise arc-disjoint $S$-cycles in $D$.
If $p=0$, the above construction reduces to just $C_1$ and $C_2$ (with no $C_j$ terms), using the arcs available in the $p=0$ construction. 
In either case, these cycles are pairwise arc-disjoint $S$-cycles, so $\lambda_S^c(D)\geq p+2$.

On the other hand, assume that there exists $p+2$ arc-disjoint $S$-cycles in $D$, denoted by \(\{C_i \mid i \in [p+2]\}\). The equalities $\deg^+_D(x_i)=\deg^-_D(x_i)=p+2$ for each $i\in[k]$, which holds in both cases $p\geq 1$ and $p=0$, imply that every $C_i$ includes an $x_{k-1}-x_k$ path and an $x_k-x_1$ path. Consequently, we may select an $x_{k-1}-x_{k}$ path, denote it by $Q_1$, that contains $s_1$ and $t_1$ in some $C_i$, and an $x_{k}-x_{1}$ path, denote it by $Q_2$, that contains $s_2$ and $t_2$ in some $C_j$. When $i=j$, one can easily find a pair of arc-disjoint $s_1-t_1$ path and $s_2-t_2$ path contained within the $S$-cycle $C_i$ in $G$. When $i\neq j$, since $C_i$ and $C_j$ are arc-disjoint, $Q_1$ and $Q_2$ are also arc-disjoint $s_1-t_1$ and $s_2-t_2$ paths in $G$. Therefore, in both cases, $G$ contains a pair of arc-disjoint $s_1-t_1$ path and $s_2-t_2$ path, as desired.\qedhere

\end{proof}

Next, we consider the complexity of $\lambda_S^c(D)$ on planar digraphs. To prove the hardness, we reduce from the {\sc Planar Arc-disjoint Paths (with Two Demand Pairs on Outer Face)} problem, which is formulated as follows: let $D=(V,A)$ be a planar digraph and $(s_1,t_1), (s_2, t_2)$ be two demand pairs such that all terminals of the demand pairs lie on the outer face of a fixed planar embedding of $D$, appearing in the cyclic order $s_1,s_2,t_1,t_2$ along the outer face. The goal is to find $d_1$ $s_1-t_1$ paths, and $d_2$ $s_2-t_2$ paths such that all of these $d_1+d_2$ paths are arc-disjoint. We use $[G; s_1,t_1,d_1; s_2,t_2,d_2]$ to denote an instance of the {\sc Planar Arc-disjoint Paths} problem. In 2012, Naves~\cite{Naves} showed that this problem is NP-complete. Thus, by our reduction, we obtain the result stated in Theorem~\ref{planar digraph}.


\begin{theorem}\label{Mp1}~\cite{Naves}\label{planar arc-disjoint paths}
The problem of {\sc Planar Arc-disjoint Paths (with Two Demand Pairs on Outer Face)} is NP-complete.
\end{theorem}


\begin{theorem}\label{planar digraph}
Fix an integer $k \geq 2$ and let $\ell$ be part of the input. Let $D$ be a planar digraph and $S\subseteq V(D)$ with $|S|=k$. The problem of deciding whether \(\lambda _S^c(D) \geq \ell\) is NP-complete.
\end{theorem}

\begin{proof}
It is easy to see that the problem is in NP. We will show that the problem is NP-hard by a reduction from the problem of {\sc Planar Arc-disjoint Paths (with Two Demand Pairs on Outer Face)}, as the latter one is NP-complete by Theorem~\ref{Mp1}.

Let $[G; s_1,t_1,d_1; s_2,t_2,d_2]$ be any instance of the problem of {\sc Planar Arc-disjoint Paths (with Two Demand Pairs on Outer Face)}. Let  \(S=\{x_i \mid i \in [k]\}\). We construct  a new planar digraph as follows. Let \[V(D')=V(G) \cup S, \] and let \[A(D')=A(G) \cup \{\overrightarrow { x_ix_{i+1}},  \overrightarrow { x_{i+1}x_i}\mid i\in [k-1]\} \cup \{\overrightarrow {x_1s_1}, \overrightarrow {x_ks_2}, \overrightarrow {t_2x_1}, \overrightarrow {t_1x_k}\}.\] Furthermore, we add \(d_1-1\) copies of arcs \(\overrightarrow{x_{i+1}x_i}\) for each \(i\in[k-1]\), \(\overrightarrow{x_1s_1}\), and \(\overrightarrow{t_1x_k}\), and \(d_2-1\) copies of arcs \(\overrightarrow{x_i x_{i+1}}\) for each \(i\in[k-1]\), \(\overrightarrow{x_k s_2}\), and \(\overrightarrow{t_2 x_1}\). To avoid parallel arcs, for each $a\in[d_1]$, we insert a new vertex $q_{i+1,i}^a$ on each $\overrightarrow {x_{i+1}x_i}$, a vertex $e_a$ on each $\overrightarrow {x_1s_1}$, and a vertex $e_a'$ on each $\overrightarrow {t_1x_k}$. Similarly, for each $b\in[d_2]$, we insert a new vertex $p_{i,i+1}^b$ on each $\overrightarrow {x_ix_{i+1}}$, a vertex $f_b$ on each $\overrightarrow {x_ks_2}$, and a vertex $f_b'$ on each $\overrightarrow {t_2x_1}$. Let $D$ denote the resulting digraph illustrated in Figure~\ref{figure3}, which can be readily checked to be planar.

\begin{figure}[htb]
    \centering 
    \makebox[0pt]{
        \begin{tikzpicture}[scale=1.0, >=latex, every node/.style={font=\small}]
    
            \filldraw[black]    (0, 0)  circle (2pt)  node [anchor=north east] {$x_{1}$};
            \filldraw[black]    (1.8, 0)  circle (2pt)  node [anchor=north] {$x_{2}$};
            \filldraw[black]    (3.6, 0)  circle (2pt)  node [anchor=north, xshift=-4pt] {$x_{k-2}$};
            \filldraw[black]    (5.4, 0)  circle (2pt)  node [anchor=north, yshift=-4pt] {$x_{k-1}$};
            \filldraw[black]    (7.2, 0)  circle (2pt)  node [anchor=north west] {$x_{k}$};

            \filldraw[black] (0.9, 0.2025) circle (2pt) node[anchor=south] {$p_{1,2}^b$};
            \draw[thick, -latex] plot[smooth, domain=0:0.9] (\x, {-0.25*\x*(\x-1.8)});
            \draw[thick, -latex] plot[smooth, domain=0.9:1.8] (\x, {-0.25*\x*(\x-1.8)});

            \filldraw[black] (4.5, 0.2025) circle (2pt) node[anchor=south] {$p_{k-2,k-1}^b$};
            \draw[thick, -latex] plot[smooth, domain=3.6:4.5] (\x, {-0.25*(\x-3.6)*(\x-5.4)});
            \draw[thick, -latex] plot[smooth, domain=4.5:5.4] (\x, {-0.25*(\x-3.6)*(\x-5.4)});

            \filldraw[black] (6.3, 0.2025) circle (2pt) node[anchor=south] {$p_{k-1,k}^b$};
            \draw[thick, -latex] plot[smooth, domain=5.4:6.3] (\x, {-0.25*(\x-5.4)*(\x-7.2)});
            \draw[thick, -latex] plot[smooth, domain=6.3:7.2] (\x, {-0.25*(\x-5.4)*(\x-7.2)});

            \filldraw[black] (0.9, -0.2025) circle (2pt) node[anchor=north] {$q_{2,1}^a$};
            \draw[thick, -latex] plot[smooth, domain=1.8:0.9] (\x, {0.25*\x*(\x-1.8)});
            \draw[thick, -latex] plot[smooth, domain=0.9:0] (\x, {0.25*\x*(\x-1.8)});

            \filldraw[black] (4.5, -0.2025) circle (2pt) node[anchor=north] {$q_{k-1,k-2}^a$};
            \draw[thick, -latex] plot[smooth, domain=5.4:4.5] (\x, {0.25*(\x-3.6)*(\x-5.4)});
            \draw[thick, -latex] plot[smooth, domain=4.5:3.6] (\x, {0.25*(\x-3.6)*(\x-5.4)});

            \filldraw[black] (6.3, -0.2025) circle (2pt) node[anchor=north, xshift=3pt] {$q_{k,k-1}^a$};
            \draw[thick, -latex] plot[smooth, domain=7.2:6.3] (\x, {0.25*(\x-5.4)*(\x-7.2)});
            \draw[thick, -latex] plot[smooth, domain=6.3:5.4] (\x, {0.25*(\x-5.4)*(\x-7.2)});

            \def\boxleft{1.2}    
            \def\boxright{6}   
            \def\boxtop{-2.0}     
            \def\boxbottom{-3.6}  
            \draw[line width=1pt] (\boxleft, \boxtop) rectangle (\boxright, \boxbottom) node[anchor=west] {$G$};
            
            \draw[densely dotted, line width=1.2pt] (2.2, 0) -- (3.2, 0);

            \filldraw[black] (2, -2.5) circle (2pt) node[anchor=south] {$s_1$};
            \filldraw[black] (5.2, -2.5) circle (2pt) node[anchor=south] {$s_2$};
            \filldraw[black] (2, -3.1) circle (2pt) node[anchor=north] {$t_2$};
            \filldraw[black] (5.2, -3.1) circle (2pt) node[anchor=north] {$t_1$};

            \coordinate (ea) at (0.7, -1.35); 
            \draw[thick, -latex] (0,0) to[bend right=10] (ea);
            \draw[thick, -latex] (ea) to[bend right=15] (2,-2.5);
            \filldraw[black] (ea) circle (2pt) node[anchor=west] {$e_a$};

            \coordinate (fb') at (0.6, -2);  
            \draw[thick, -latex] (fb') to[bend right=-10] (0,0);
            \draw[thick, -latex] (2,-3.1) to[bend right=-15] (fb');
            \filldraw[black] (fb') circle (2pt) node[anchor=east] {$f_b'$};

            \coordinate (ea') at (6.6, -2); 
            \draw[thick, -latex] (ea') to[bend right=10] (7.2, 0);
            \draw[thick, -latex] (5.2,-3.1) to[bend right=15] (ea');
            \filldraw[black] (ea') circle (2pt) node[anchor=
            west] {$e_a'$};

            \coordinate (fb) at (6.4, -1.35);  
            \draw[thick, -latex] (7.2,0) to[bend right=-12] (fb);
            \draw[thick, -latex] (fb) to[bend right=-15] (5.2, -2.5);
            \filldraw[black] (fb) circle (2pt) node[anchor=east] {$f_b$};
        \end{tikzpicture}
        }
    \caption{The planar digraph $D$.}
    \label{figure3}
\end{figure}
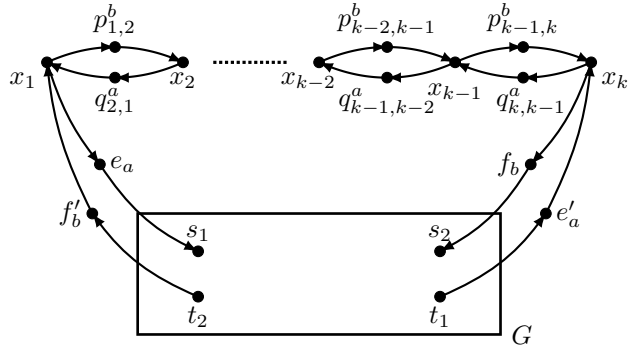

We will show that $\lambda_S^c(D) \geq \ell = d_1 + d_2$ if and only if there exist $d_1$ $s_1-t_1$ paths $P_a$ for $a\in [d_1]$ and $d_2$ $s_2-t_2$ paths $Q_b$ for $b\in [d_2]$ in $G$ that are arc-disjoint.

If there exist $d_1$ $s_1-t_1$ paths $P_a$ for $a\in [d_1]$ and $d_2$ $s_2-t_2$ paths $Q_b$ for $b\in [d_2]$ which are pairwise arc-disjoint in $G$, then \[C_a=x_1e_as_1P_at_1e_a'x_kq_{k,k-1}^ax_{k-1} \dots x_2q_{2,1}^ax_1, a\in [d_1],\] \[C_b'=x_1p_{1,2}^bx_2 \dots x_{k-1} p_{k-1,k}^b x_k f_bs_2Q_bt_2f_b'x_1, b\in [d_2]\] are $d_1+d_2$ arc-disjoint $S$-cycles in $D$. Thus, \(\lambda _S^c(D) \geq \ell\).

Conversely, suppose there exist $d_1+d_2$ arc-disjoint $S$-cycles in $D$, denoted by $C_i$ for $i\in[d_1+d_2]$. Since each $S$-cycle contains both an $x_1-x_k$ path and an $x_k-x_1$ path, the $d_1+d_2$ $S$-cycles provide $d_1+d_2$ paths of each type. On the other hand, since $\deg_D^+(x_1)=\deg_D^-(x_1)=d_1+d_2$, all arcs incident with $x_1$ are exactly those used by these paths, and among these paths, exactly $d_1$ are $s_1-t_1$ paths and $d_2$ are $s_2-t_2$ paths that pass through $G$. Moreover, each such $x_1-x_k$ path must contain an $s_1-t_1$ path, since each crossing arc of $G$ (the arc with exactly one end-vertex in $V(G)$) leaves from $t_1$ or $t_2$, and the crossing arc from $t_2$ can only return to $x_1$, which cannot form an $S$-cycle. Similarly, each of the $d_2$ $x_k-x_1$ paths must contain an $s_2-t_2$ path. Therefore, we obtain $d_1$ $s_1-t_1$ paths and $d_2$ $s_2-t_2$ paths that are pairwise arc-disjoint in $G$. 

By Theorem~\ref{planar arc-disjoint paths}, for fixed $k\geq2$ and $\ell$ is part of the input, the problem of deciding whether $\lambda_S^c(D)\geq \ell$ on planar digraphs is NP-complete.
\end{proof}

The {\sc Hamiltonian cycle} problem is a classic problem in graph theory and it is defined as follows: given a graph $G$, the aim is to find a Hamiltonian cycle in $G$. It is known that the {\sc Hamiltonian cycle} problem is NP-complete for general undirected graphs~\cite{Karp}, planar graphs~\cite{M. R. Garey} and $d$-regular graphs where $d\geq 3$~\cite{C. Picouleau}. Using an appropriate reduction from the {\sc Hamiltonian cycle} problem to our problem, we can transform an instance of the {\sc Hamiltonian cycle} problem into an instance of our problem, thereby proving that deciding $\lambda_S^c(D) \ge \ell$ remains NP-complete on symmetric digraphs, planar digraphs and Eulerian digraphs when $\ell$ is fixed and $k$ is part of the input.

\begin{theorem}~\cite{Karp}\label{Hamiltonian cycle problem}
The {\sc Hamiltonian cycle} problem in undirected graphs is NP-complete.
\end{theorem}

We now state the theorem for symmetric digraphs (where $k$ is part of the input) and provide its proof.


\begin{theorem}\label{symmetric digraph}
Fix an integer $\ell \geq 1$ and let $k$ be part of the input. For a symmetric digraph $D$ and a subset $S\subseteq V(D)$ with $|S|=k$, deciding whether $\lambda_S^c(D)\geq \ell$ is NP-complete.
\end{theorem}

\begin{proof}
Let $G=(V(G),E(G))$ be an undirected graph of order $n$ with vertex set $V(G)=\{x_i\mid i\in[n]\}$.
We construct a symmetric digraph $D$ as follows. For each $e=x_ix_j\in E(G)$, we add $\ell-1$ additional copies. To avoid parallel edges, we insert $v_{i,j}^a$  ($a\in [\ell]$) on each edge $x_ix_j$. Then we replace each edge with two arcs in opposite directions. That is, we replace each edge $x_ix_j$ in $G$ by the structure shown on the right side of Figure~\ref{figure4}. Clearly, the construction of $D$ can be completed in polynomial time. 

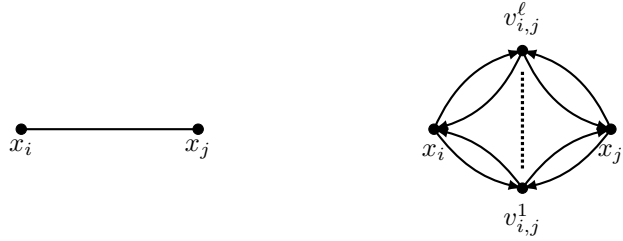
\begin{figure}[htb]
    \centering 
    \makebox[0pt]{
        \begin{tikzpicture}[scale=1.3, >=latex, every node/.style={font=\small}]
    
            \filldraw[black] (0,0)   circle (1.5pt) node [anchor=north] {$x_{i}$};
            \filldraw[black] (1.8,0) circle (1.5pt) node [anchor=north] {$x_{j}$};
            \draw [thick] (0,0) -- (1.8,0);
     
            \filldraw[black] (4.2,0) circle (1.5pt) node [anchor=north, yshift=-2pt] {$x_{i}$};
            \filldraw[black] (6.0,0) circle (1.5pt) node [anchor=north, yshift=-2pt] {$x_{j}$};
      
            \filldraw[black] (5.1, -0.6) circle (1.5pt) node [anchor=north, yshift=-2pt] {$v_{i,j}^1$};
            \filldraw[black] (5.1,  0.8) circle (1.5pt) node [anchor=south, yshift=2pt] {$v_{i,j}^{\ell}$};

            \draw[thick, -latex] (4.2,0) to [bend left=25] (5.1,0.8);
            \draw[thick, -latex] (5.1,0.8) to [bend left=25] (4.2,0);
            \draw[thick, -latex] (6.0,0) to [bend right=25] (5.1,0.8);
            \draw[thick, -latex] (5.1,0.8) to [bend right=25] (6.0,0);

            \draw[thick, -latex] (4.2,0) to [bend right=20] (5.1,-0.6);
            \draw[thick, -latex] (5.1,-0.6) to [bend right=20] (4.2,0);
            \draw[thick, -latex] (6.0,0) to [bend left=20] (5.1,-0.6);
            \draw[thick, -latex] (5.1,-0.6) to [bend left=20] (6.0,0);
            
            \draw[densely dotted, line width=1.2pt]      (5.1, -0.4) -- (5.1, 0.6);

        \end{tikzpicture}
        }
    \caption{An edge $x_ix_j \in E(G)$ and its replacement structure in the construction of $D$}
    \label{figure4}
\end{figure}

Let $S=V(G)$. Now we prove that $G$ has a Hamiltonian cycle if and only if $D$ contains $\ell$ arc-disjoint $S$-cycles. First assume $G$ has a Hamiltonian cycle $C = x_1x_2\cdots x_nx_1$. Then it is easy to verify that the cycles
\[C_a=x_1v_{1,2}^ax_2\cdots x_{n-1}v_{n-1,n}^ax_nv_{n,1}^ax_1, a\in [\ell]\]
are $\ell$ arc-disjoint $S$-cycles in $D$. Conversely, if there exist $\ell$ arc-disjoint $S$-cycles in $D$, then choose any such $S$-cycle, denoted by $C'$. Since $S = V(G) \subseteq V(C')$, $C'$ must pass through every vertex from $V(G)$, with each such vertex appearing exactly once. Moreover, by the construction of $D$, for any pair of vertices $x_i, x_j\in V(G)$,  $\overrightarrow {x_ix_j}, \overrightarrow {x_jx_i} \not\in A(D)$.
Instead, one can only go from $x_i$ (resp.\ $x_j$) to $x_j$ (resp.\ $x_i$) via some vertex $v_{x_ix_j}^a$. Therefore, by replacing every 2-path $x_iv_{ij}^ax_j$ or $x_jv_{ij}^ax_i$ in $C'$ with the corresponding undirected edge $x_ix_j$ of $G$, we obtain a Hamiltonian cycle in $G$. 

Now by Theorem~\ref{Hamiltonian cycle problem}, when $\ell \geq 1 $ and $k$ is part of the input, the problem of deciding whether \(\lambda _S^c(D) \geq \ell\) is NP-complete. \qedhere

\end{proof}

For symmetric digraphs, although the problem is NP-complete when $\ell$ is fixed and $k$ is part of the input, it becomes polynomial-time solvable when $k$ is fixed and $\ell \in \{1,2\}$. We now give the proof.

\begin{theorem}\label{symmetric digraph2}
Let $k \geq 2$ be a fixed integer, $\ell \in\{ 1,2\}$. Let $D$ be a symmetric digraph and $S\subseteq V(D)$ with $|S|=k$. The problem of deciding whether \(\lambda _S^c(D) \geq \ell \) can be solved in polynomial time.
\end{theorem}

\begin{proof}
 Observe that $\kappa_S^c(D) \geq 1$ if and only if $\lambda_S^c(D) \geq 1$, by Theorem~3.6 in \cite{Sun-Jin1}, the problem of deciding whether $\kappa_S^c(D) \geq 1$ is solvable in polynomial time. Thus the problem of deciding whether $\lambda_S^c(D)\geq 1$ is polynomial-time solvable.

We now turn to the case $\ell=2$ and analyze the problem by considering two cases separately: $k \geq 3$ and $k = 2$.

We first consider the case that $k \geq 3$. Suppose that $C$ is an $S$-cycle in $D$. Replacing each arc of $C$ with its reverse arc results in another $S$-cycle which is arc-disjoint from $C$, and this directly implies $\lambda_S^c(D) \geq 2$. Hence, $D$ has two arc-disjoint $S$-cycles whenever it has one. Therefore $\lambda_S^c(D)\ge 2$ if and only if $\lambda_S^c(D)\ge 1$, so the problem is polynomial-time solvable by the preceding paragraph.

It remains to consider the case that $k=2$. Let $S=\{u,v\}$. If there exist two internally disjoint $(u,v)$-paths in the underlying undirected graph $G$ of $D$, then we can orient these paths appropriately to yield two arc-disjoint $S$-cycles in $D$. Conversely, suppose $D$ contains two arc-disjoint $S$-cycles. Since they are arc-disjoint, they cannot both be the $S$-cycle of length 2, consisting of the two arcs $\overrightarrow{uv}$ and $\overrightarrow{vu}$. Thus at least one of them has length more than 2. Choose such an $S$-cycle. Removing its orientation yields two internally disjoint $(u,v)$-paths in the underlying undirected graph $G$. 
Since finding two internally disjoint $(u,v)$-paths in the undirected graph $G$ can be solved in polynomial time, the problem of deciding whether \(\lambda _S^c(D) \geq 2\) can also be solved in polynomial time.   
\end{proof}

\begin{theorem}\label{PlaHamil}\cite{M. R. Garey}
    The undirected {\sc planar Hamiltonian cycle} problem is NP-complete.
\end{theorem}

Note that the construction in the proof of Theorem~\ref{symmetric digraph}  preserves planarity. Therefore, if $G$ is an undirected planar graph, then a directed planar graph $D$ can be constructed in the same manner. Similarly, we can prove that $G$ has a Hamiltonian cycle if and only if there exist $\ell$ arc-disjoint $S$-cycles where $S=V(G)$ in $D$. By Theorem~\ref{PlaHamil}, we immediately obtain Theorem~\ref{planar digraph2} (the proof is analogous to that of Theorem~\ref{symmetric digraph}, we omit the details here).

\begin{theorem}\label{planar digraph2}
Let $\ell\geq 1$ be a fixed integer and let $k$ be part of the input. Then, for a planar digraph $D$ and a set $S\subseteq V(D)$ with $|S|=k$, the decision problem of whether $\lambda_S^c(D)\geq \ell$ is NP-complete.

\end{theorem}

\begin{theorem}~\cite{C. Picouleau}\label{$d$-regular graph}
    For any fixed $d \geq 3$, the problem of deciding whether a $d$-regular graph has a Hamiltonian cycle is NP-complete. 
\end{theorem}

Given a connected $d$-regular undirected graph $G$, we apply the same replacement construction as in the proof of Theorem~\ref{symmetric digraph} (see Figure~\ref{figure4}): each edge of $G$ is replaced by $\ell$ internally subdivided paths, and then each undirected edge is replaced by two opposite arcs. Let $D$ be the resulting digraph and let $S=V(G)$. In $D$, every vertex originally from $G$ has in-degree and out-degree $\ell d$, whereas each subdivision vertex has in-degree and out-degree $2$. Since $G$ is connected, $D$ is Eulerian.
The same argument as in Theorem~\ref{symmetric digraph} shows that $G$ has a Hamiltonian cycle if and only if $D$ has $\ell$ arc-disjoint $S$-cycles. Therefore, by Theorem~\ref{$d$-regular graph}, for fixed $\ell\ge 1$ and $k$ part of the input, deciding whether $\lambda_S^c(D)\geq \ell$ on Eulerian digraphs is NP-complete. This completes the proof of Theorem~\ref{Eulerian digraph2}.

\begin{theorem}\label{Eulerian digraph2}
Let $\ell\geq 1$ be a fixed integer and let $k$ be part of the input. Then, for an Eulerian digraph $D$ and a set $S\subseteq V(D)$ with $|S|=k$, the decision problem of whether $\lambda_S^c(D)\geq \ell$ is NP-complete.

\end{theorem}

Theorems~\ref{Eulerian digraph}, \ref{planar digraph} and \ref{symmetric digraph} immediately imply the following corollary, where $k$ is also part of the input.

\begin{corollary}
Let $k$ and $\ell$ be part of the input. Deciding whether $\lambda_S^c(D) \geq \ell$ where $S \subseteq V(D)$ with $|S| = k$ is NP-complete for Eulerian digraphs, planar digraphs and symmetric digraphs.
\end{corollary}

\section{Exact values for $\lambda_{k}^{c}(D)$ on several classes of digraphs}

The main result of this section is a precise characterization of the directed cycle $k$-arc-connectivity for three special classes of digraphs: complete digraph $\overleftrightarrow{K}_{n}$, complete regular multipartite digraph $\overleftrightarrow{K}_{[w]^l}$, and complete bipartite digraph $\overleftrightarrow{K}_{t,z}$.


By the definition of $\lambda_{k}^{c}(D)$, we directly obtain the following two facts.

\begin{Fact}\label{Fact1}
Suppose $D$ is a digraph of order $n$. For all integers $2\leq k\leq n$, we have $\lambda_{k}^{c}(D)\leq \delta^{0}(D)$.

\end{Fact}

\begin{Fact}\label{Fact2}
Suppose $D$ is a digraph of order $n$. Then $\lambda_{k+1}^c(D)\le\lambda_k^c(D)$ for all integers $2\leq k\leq n$.


\end{Fact}
By Fact~\ref{Fact2}, the parameter $\lambda_k^c(D)$ is a non-increasing function of $k$. To determine $\lambda_k^c(\overleftrightarrow{K}_n)$, we further need the following famous theorem due to Tillson:
\begin{theorem}[Tillson's Decomposition Theorem]\cite{{Tillson T W}} \label{Tillson's}
The complete digraphs $\overleftrightarrow{K}_{n}$ can be decomposed into Hamiltonian cycles if and only if $n\ne 4,6$.
\end{theorem} 

We now derive the exact value of $\lambda_{k}^c(\overleftrightarrow{K}_{n})$ and present its proof.

\begin{theorem}\label{1}
For $\overleftrightarrow{K}_n$, the values of $\lambda_k^c$ are as follows:
\begin{description}
\item[(i) ]When $n\ne 4,6$, we have $\lambda _{k}^{c} (\overleftrightarrow{K}_{n})=n-1$.
\item[(ii) ]When $n=4,6$, we have $\lambda _{k}^{c}(\overleftrightarrow{K}_{n})=\left\{\begin{matrix}
 n-1,  & 2\leq k\leq n-2, \\
  n-2, & n-1 \leq k\le n.
\end{matrix}\right. $
\end{description}
\end{theorem}
\begin{proof}
We first prove part \textbf{(i)}.  Let $D\cong \overleftrightarrow{K}_{n}$ and  $S\subseteq V(D)$ with $|S|=k$. By Theorem~\ref{Tillson's}, the arcs of $D$ can be decomposed into $n-1$ Hamiltonian cycles and so $\lambda_{n}^{c}(D)\ge n-1$. Therefore, $\lambda_{k}^{c}(D)\ge \lambda_{n}^{c}(D)\ge n-1$ by Fact~\ref{Fact2}. On the other hand, $\lambda_{k}^{c}(D)\le n-1$ due to Fact~\ref{Fact1}.  Therefore, $\lambda_{k}^{c}(D)= n-1$, as desired.

We next prove part \textbf{(ii)}. Let $D\cong \overleftrightarrow{K}_{n}$ and $V(D)=\{u_i\mid i\in [n]\}$. We claim that for $n=4,6$, $\lambda_n^c(D)=n-2$.  When $n=4$, $C_1=u_1u_2u_3u_4u_1$ and $C_2=u_1u_4u_3u_2u_1$ are two arc-disjoint $S$-cycles in $D$, thus  $\lambda _{n}^{c}(D)\ge2=n-2$. When $n=6$, $C_1=u_1u_2u_3u_4u_5u_6u_1$, $C_2=u_1u_6u_5u_4u_3u_2u_1$, $C_3=u_1u_3u_6u_4u_2u_5u_1$ and $C_4=u_1u_5u_2u_4u_6u_3u_1$ are four arc-disjoint $S$-cycles in $D$. Hence, $\lambda _{n}^{c}(D)\ge4=n-2$. Furthermore, by Theorem~\ref{Tillson's}, we ascertain that when $n=4,6$, $\lambda _{n}^{c}(D)\ne n-1$, which implies that $\lambda _{n}^{c}(D)\le n-2$.  Thus $\lambda _{n}^{c}(D)=n-2$ for $n=4,6$. We now handle the remaining cases.

\begin{case 1}$n=4$ and $k\in \{2,3\}$. The cases $k=2$ for $n=4$ are simple, so we only discuss $k=3$ in detail.


Without loss of generality, assume that $S=\{u_1,u_2,u_3\}\subseteq V(D)$.
Clearly, $C_1=u_1u_2u_3u_1$, $C_2=u_1u_3u_2u_1$ are two arc-disjoint $S$-cycles in $D$. Hence, $\lambda _{k}^{c}(D)\geq 2$. 

Next, we prove $\lambda _{k}^{c}(D)\leq 2$. Assume that $\lambda_k^c(D) \ge 3$. Then there exists $S = \{u_1, u_2, u_3\} \subseteq V(D)$ such that $D$ contains three arc-disjoint $S$-cycles. Every vertex $u_i\in S$ is satisfied $deg_D^+(u_i)=deg_D^-(u_i)=3$, and each $S$-cycle uses exactly one out-arc and one in-arc at $u_i$. Thus the three $S$-cycles together consume all six arcs incident with the fourth vertex $u_4$ (the arcs $\overrightarrow{u_iu_4}$ and $\overrightarrow{u_4u_i}$ for $i \in [3]$), forcing all three $S$-cycles to contain $u_4$. Consequently, we have three arc-disjoint Hamiltonian cycles in $D$, which contradicts Theorem~\ref{Tillson's}. Hence $\lambda _{k}^{c}(D)\leq 2$, and we conclude $\lambda _{k}^{c}(D) = 2$.

\end{case 1}

\begin{case 2}$n=6$ and $k\in\{2,3,4,5\}$. The cases $k=2,3$ for $n=6$ are simple, so we only discuss $k=4,5$ in detail.


\begin{subcase 2.1} When $n=6$ and $k=4$, without loss of generality,  assume that $S=\{u_1,u_2,u_3,u_4\}\subseteq V(D)$. We can find $C_1=u_1u_2u_3u_4u_5u_1, C_2=u_1u_3u_6u_2u_4u_1,C_3=u_1u_4u_3u_5u_2u_6u_1,C_4=u_1u_5u_4u_6u_3u_2u_1, C_5=u_1u_6u_4u_2\\u_5u_3u_1$ are five arc-disjoint $S$-cycles in $D$. Thus $\lambda_k^c(D) \geq 5$. By Fact ~\ref{Fact2}, we have $\lambda_k^c(D) \leq 5$. Therefore, $\lambda_k^c(D) = 5$.

\end{subcase 2.1}

\begin{subcase 2.2} When $n=6$ and $k=5$, we may assume $S=\{u_1,u_2,u_3,u_4,u_5\}$ and let $u_6$ be the remaining vertex. $C_1=u_1u_2u_3u_4u_5u_1$, $C_2=u_1u_3u_5u_2u_4u_1$, $C_3=u_1u_4u_2u_5u_3u_1$ and $C_4=u_1u_5u_4u_3u_2u_1$ are four arc-disjoint $S$-cycles in $D$. Hence $\lambda _{k}^{c}(D)\geq 4$. 

Suppose, for contradiction, that $\lambda_k^c(D) \geq 5$. Then there exist five arc-disjoint $S$-cycles. For every $u_i \in V(D)$, we have $deg_D^+(u_i)=deg_D^-(u_i)=5$. These five $S$-cycles consume exactly one out-arc and one in-arc at each $u_i$, hence they use all ten arcs incident with $u_6$ (the arcs $\overrightarrow{u_iu_6}$ and $\overrightarrow{u_6u_i}$ for $i\in[5]$), forcing all five $S$-cycles to contain $u_6$. This means we have five arc-disjoint Hamiltonian cycles in $D$, which contradicts Theorem~\ref{Tillson's}. Hence $\lambda_k^c(D) \leq 4$, and  we conclude $\lambda_k^c(D) = 4$. \qedhere
\end{subcase 2.2}

\end{case 2} 
\end{proof}




We now consider the digraphs $\overleftrightarrow{K}_{[w]^l}$ and $\overleftrightarrow{K}_{t,z}$. To determine their directed cycle $k$-arc-connectivity, we first recall a classical result of Ng on Hamiltonian cycle decomposition for $\overleftrightarrow{K}_{[w]^l}$ (Theorem~\ref{Ng L L}), which will be used in our proofs for both families. Then we state our main theorem for these digraphs (Theorem~\ref{2}) and present its proof.

\begin{theorem}\label{Ng L L}\cite{{Ng L L}} 
The digraph $\overleftrightarrow{K}_{w,w,\dots,w } $ ($l$ times)(simply,  $\overleftrightarrow{K}_{[w]^l}$)  is Hamiltonian decomposable if and only if $(l,w)\neq (4,1) $ and $(l,w)\neq (6,1)$.
\end{theorem} 

\begin{theorem}\label{2}
For $\overleftrightarrow{K}_{[w]^l}$, the values of $\lambda_k^c$ are as follows:
\begin{description}
\item[(i) ]When $(l,w)\notin \{(4,1), (6,1)\}$, we have $\lambda _{k}^{c}(\overleftrightarrow{K}_{[w]^l})=w(l-1)$.
\item[(ii) ]When $2\le t<z$, we have $\lambda _{k}^{c}(\overleftrightarrow{K}_{t,z}) = \left\{\begin{matrix}
 t, & 2\le k\le t, \\
 0, & t+1\le k\le t+z.
\end{matrix}\right. $
\end{description}
\end{theorem}
\begin{proof} We first prove part \textbf{(i)}. Let $D\cong\overleftrightarrow{K}_{[w]^l}$ and $S\subseteq V(D)$ with $|S|=k$, where  $(l,w) \notin \{(4,1), (6,1)\}$. On the one hand, by Theorem~\ref{Ng L L}, $D$ can be decomposed into $w(l-1)$ Hamiltonian cycles, which are also $w(l-1)$ arc-disjoint $S$-cycles. Hence, $\lambda _{k}^{c}(D)\ge w(l-1)$. On the other hand, by Fact~\ref{Fact1}, $\lambda_{k}^{c}(D)\le \delta^{0}(D)=w(l-1)$.  Therefore, $\lambda_{k}^{c}(D)= w(l-1)$, as desired.


We next prove part \textbf{(ii)}. Let $D\cong\overleftrightarrow{K}_{t,z}$ with two parts $X=\left \{ x_{1} ,x_{2},\dots ,x_{t} \right \} $ and $Y=\left \{ y_{1} ,y_{2},\dots ,y_{z} \right \}$, where $2\le t<z$. The proof will be completed by the following claims.
\vspace{2mm}

\noindent{\bf Claim 1:} 
When $2\le k\le t$, we have $\lambda _{k}^{c}(D)=t $.   

\vspace{2mm}

\noindent{\bf Proof of Claim 1:} Let $S\subseteq V(D)$ with $|S|=k$. We will complete the proof by considering the distribution of vertices in $S$.

\begin{case 1}$S\subseteq X$ or $S\subseteq Y$.

If $S\subseteq X$, without loss of generality, assume $S=\{x_1,\dots,x_k\}$. Since $k\leq t < z$, there exists a copy of $\overleftrightarrow{K}_{t,t}$ in $D$ that contains all vertices of $S$. By Theorem~\ref{Ng L L}, this copy can be decomposed into $t$ Hamiltonian cycles, which are $t$ arc-disjoint $S$-cycles. The case $S\subseteq Y$ is analogous.

Hence, in this case we have $\lambda _{S}^{c}(D)\geq t$.


\end{case 1}
\begin{case 2}$S\cap X\ne \emptyset, S\cap Y\ne \emptyset $.

Without loss of generality, assume $S=\{x_1,\dots,x_p,y_1,\dots,y_{k-p}\}$, where $p < k$. Since $k\le t < z$, there exists a copy of $\overleftrightarrow{K}_{t,t}$ in $D$ that contains all vertices of $S$. By Theorem~\ref{Ng L L}, this copy can be decomposed into $t$ Hamiltonian cycles, which are $t$ arc-disjoint $S$-cycles.

\end{case 2}
Now we have  $\lambda _{S}^{c}(D)\ge t$ for each $S\subseteq V(D)$ and so $\lambda_{k}^{c}(D)\ge t$. Furthermore, we conclude that $\lambda _{k}^{c}(D)=t$ by Fact~\ref{Fact1}. 
\qed

\noindent{\bf Claim 2:} When $t+1\le k\le t+z$, we have $\lambda _{k}^{c}(D)=0 $.

\vspace{2mm}

\noindent{\bf Proof of Claim 2:} For $t+1\leq k\leq z$, take $S'=\{y_1,\dots,y_k\}$. For $z<k\le t+z$, take $S''=\{x_1,\dots,x_{k-z},y_1,\dots,y_z\}$ where $k-z<z$. Since every cycle in $\overleftrightarrow{K}_{t,z}$ is even and contains an equal number of vertices from $X$ and $Y$,  there exists neither an $S'$-cycle nor an $S''$-cycle. Hence $\lambda_k^c(D)=0$.\qedhere

\end{proof}

\vskip 1cm

\noindent{\bf Acknowledgement.} This work was supported by National Natural Science Foundation of China under Grant No. 12371352 and Yongjiang Talent Introduction Programme of Ningbo under Grant No. 2021B-011-G. 

\vskip 3mm
\noindent {\bf Data availability} No data was used for the research described in the article.


\begin{thebibliography}{20}


\bibitem{Bang-Jensen09} J. Bang-Jensen and G. Gutin, Digraphs: Theory, Algorithms and Applications, 2nd Edition, Springer, London, 2009.

\bibitem{J.A. Bondy} J.A. Bondy and U.S.R. Murty, Graph Theory, Springer, Berlin, 2008.

\bibitem{Cheriyan-Salavatipour} J. Cheriyan and M. Salavatipour, Hardness and approximation results for packing Steiner trees, Algorithmica, 45, 2006, 21--43.


\bibitem{S. Fortune} S. Fortune, J. Hopcroft and J. Wyllie, The directed subgraph homeomorphism problem, Theoret. Comput. Sci., 10, 1980, 111--121.

\bibitem{Grotschel-Martin-Weismantel} M. Gr\"{o}tschel, A. Martin and R. Weismantel, The Steiner tree packing problem in VLSI design, Math. Program., 78, 1997, 265--281.

\bibitem{M. R. Garey} M. R. Garey, D. S. Johnson, and R. E. Tarjan, The planar Hamiltonian circuit problem is NP-complete, SIAM J. Comput., 5(4), 1976, 704--714.


 \bibitem{Karp} R. M. Karp, Reducibility among combinatorial problems, Complexity of Computer Computations, R. E. Miller and J. W. Thatcher, eds., Plenum Press, New York, 1972, 85--103.


\bibitem{Li-Mao5} X. Li and Y. Mao, Generalized Connectivity of Graphs, Springer, Switzerland, 2016.


\bibitem{Ng L L} L.L. Ng, Hamiltonian decomposition of complete regular multipartite digraphs, Discrete Math., 177(1-3), 1997, 279--285.

\bibitem{Naves} G. Naves, The hardness of routing two pairs on one face, Math. Program., 131(1), 2012, 49--69.

\bibitem{C. Picouleau} C. Picouleau, Complexity of the Hamiltonian cycle in regular graph problem, Theor. Comput. Sci., 131(2), 1994, 463--473.

\bibitem{Sherwani}
N. Sherwani, Algorithms for VLSI Physical Design Automation, 3rd Edition, Kluwer Acad. Pub., London, 1999.

\bibitem{Sun-book} Y. Sun, Steiner Type Packing Problems in Digraphs, SpringerBriefs in Mathematics, Springer, Singapore, 2026.


\bibitem{Sun-Gutin2} Y. Sun and G. Gutin, Strong subgraph connectivity of digraphs, Graphs Combin., 37, 2021, 951--970.

\bibitem{Sun-Gutin-Yeo-Zhang} Y. Sun, G. Gutin, A. Yeo and X. Zhang, Strong subgraph $k$-connectivity, J. Graph Theory, 92(1), 2019, 5--18.

\bibitem{Sun-Gutin-Zhang}
Y. Sun, G. Gutin and X. Zhang, Packing strong subgraph in digraphs, Discrete Optim.,  46, 2022, Article 100745.



\bibitem{Sun-Jin1}Y. Sun and Z. Jin, Perfect out-forest problem and directed Steiner cycle packing problem, Discrete Appl. Math., 366, 2025, 201--209.

\bibitem{Sun-Yeo} Y. Sun and A. Yeo, Directed Steiner tree packing and directed tree connectivity, J. Graph Theory, 102(1), 2023, 86--106.

\bibitem{Sun-Zhang} Y. Sun and X. Zhang, Algorithmic and structural results of directed Steiner path packing and directed path connectivity, Graphs and Combin., 42(1), 2026, Article 2.



\bibitem{Tillson T W} T. W. Tillson, A Hamiltonian decomposition of $K_{2m}^{*}$, $2m\ge 8$, J. Combin. Theory Ser. B, 29(1), 1980, 68--74.

\bibitem{Wang-Sun} C. Wang and Y. Sun, Directed cycle $k$-connectivity of complete digraphs and complete regular bipartite digraphs, Discrete Appl. Math., 358, 2024, 203–213.
\end{thebibliography}
\end{document}